\documentclass[12pt,a4paper]{amsart}

\usepackage[latin1]{inputenc}
\usepackage{enumerate}
\usepackage{rotating}
\usepackage{fancybox}
\usepackage{float}
\restylefloat{table}
\usepackage{graphicx}

\usepackage{amsmath}

\usepackage{amssymb}

\usepackage{amsthm}

\usepackage[arrow, matrix, curve]{xy}

\usepackage{a4wide}

\usepackage{txfonts}

\usepackage{hyperref}

\usepackage{color}

\theoremstyle{plain}
\newtheorem{thma}{Main Theorem}
\newtheorem{thm}{Theorem}[section]
\newtheorem{prop}[thm]{Proposition}

\theoremstyle{definition}

\newtheorem{lem}[thm]{Lemma}

\newtheorem{rmk}[thm]{Remark}

\numberwithin{equation}{thm}

\DeclareMathAlphabet{\mathpzc}{OT1}{pzc}{m}{it}

\DeclareMathOperator{\Kopf}{top}
\DeclareMathOperator{\rad}{rad}

\DeclareMathOperator{\charac}{char}

\DeclareMathOperator{\modu}{mod}
\DeclareMathOperator{\Hom}{Hom}
\DeclareMathOperator{\Ext}{Ext}

\DeclareMathOperator{\soc}{soc}

\DeclareMathOperator{\id}{id}

\DeclareMathOperator{\res}{res}

\begin{document}

\title{Auslander-Reiten theory of small half quantum groups}
\author{Julian K\"ulshammer}
\address{Christian-Albrechts-Universit\"at zu Kiel, Ludewig-Meyn-Str. 4, 24098 Kiel, Germany}
\email{kuelshammer@math.uni-kiel.de}
\thanks{Supported by the D.F.G. priority program SPP1388 "Darstellungstheorie"}

\begin{abstract}
For the small half quantum groups $u_\zeta(\mathfrak{b})$ and $u_\zeta(\mathfrak{n})$ we show that the components of the stable Auslander-Reiten quiver containing gradable modules are of the form $\mathbb{Z}[\mathbb{A}_\infty]$.
\end{abstract}

\maketitle

\section*{Introduction}

For a selfinjective algebra the shape of the components of the stable Auslander-Reiten quiver is an invariant of its Morita equivalence class. In 1995 Erdmann showed (in analogy to a result by Ringel for hereditary algebras) that all components of the stable Auslander-Reiten quiver belonging to a wild block have the same tree class $\mathbb{A}_\infty$. She also gave an analogue of this result for wild local restricted enveloping algebras.\\

In 1990 Lusztig defined a quantum analogue of the restricted enveloping algebra, called the small quantum group. Its Borel and nilpotent parts were shown to have wild representation type by Feldvoss and Witherspoon for $\mathfrak{g}\neq \mathfrak{sl}_2$ in \cite{FW09, FW11} (a generalization of a result by Cibils \cite{Cib97}). In this paper we give an analogue of Erdmann's Theorem and prove:

\begin{thma}
Let $\mathfrak{g}\neq \mathfrak{sl}_2$. Let $\mathcal{C}$ be a component of the stable Auslander-Reiten quiver of $u_\zeta(\mathfrak{b})$ or $u_\zeta(\mathfrak{n})$ containing the restriction of a $u_\zeta(\mathfrak{b})U^0_\zeta(\mathfrak{g})$-module. Then $\mathcal{C}$ has tree class $\mathbb{A}_\infty$.
\end{thma}

The main ingredients in the proof are results of Scherotzke on the Auslander-Reiten quiver of a skew group algebra and an analogue of Dade's Lemma for small quantum groups provided by Drupieski.\\

Our paper is organized as follows: In Section 1 we recall the basic definitions and fix our notation. Section 2 states Kerner and Zacharia's analogue of Webb's Theorem that restricts the tree classes of non-periodic components to Euclidean and infinite Dynkin diagrams and describes the periodic components in more detail. Section 3 is concerned with the graded module category and also provides an analogue of Webb's Theorem in this case. Section 4 and 5 then exclude the Euclidean tree classes and the other infinite Dynkin tree classes, respectively.

\section{Preliminaries}

Let $k$ be an algebraically closed field of characteristic $p\geq 0$. All vector spaces will be assumed to be finite dimensional unless stated otherwise. For a general introduction to Auslander-Reiten theory we refer the reader to \cite{ARS95} or \cite{ASS06}. We denote the syzygy functor by $\Omega$, the Auslander-Reiten translation by $\tau$. For a Frobenius algebra we denote its Nakayama automorhpism by $\nu$.\\

Let $\mathfrak{g}$ be a finite dimensional simple complex Lie algebra. Denote the corresponding set of roots by $\Phi$, a chosen set of simple roots by $\Pi$, the corresponding set of positive roots by $\Phi^+$, the Coxeter number by $h$. Let $\ell>1$ be an odd integer not divisible by three if the corresponding root system is of type $\mathbb{G}_2$. Let $U_\zeta(\mathfrak{g})$ be the Lusztig form of the quantized enveloping algebra. The finite dimensional subalgebra $u_\zeta(\mathfrak{g}):=\langle E_\alpha, F_\alpha, F_\alpha, K_\alpha^{\pm 1}|\alpha\in \Pi\rangle $ is called the small quantum group (for a general introduction, see e.g. \cite{DruPhD}). It has a triangular decomposition $u_\zeta(\mathfrak{g})\cong u_\zeta(\mathfrak{n})\otimes u_\zeta^0(\mathfrak{g})\otimes u_\zeta(\mathfrak{n}^+)$ and we denote its Borel part by $u_\zeta(\mathfrak{b}):=u_\zeta(\mathfrak{n})u_\zeta^0(\mathfrak{g})$. The nilpotent and the Borel part are linked via a skew algebra construction $u_\zeta(\mathfrak{b})=u_\zeta(\mathfrak{n})* (\mathbb{Z}/(\ell))^n$, where $n=|\Pi|$. The zero part of the triangular decomposition of $U_\zeta(\mathfrak{g})$ will be denoted $U_\zeta^0(\mathfrak{g})$.\\

In \cite{FW09, FW11} Feldvoss and Witherspoon have shown that for $\mathfrak{g}\neq \mathfrak{sl}_2$ the connected algebras $u_\zeta(\mathfrak{b})$ and $u_\zeta(\mathfrak{n})$ are wild.\\

For our considerations the following analogue of Dade's Lemma provided by Drupieski is also essential:

\begin{prop}
The restriction of a $u_\zeta(\mathfrak{b})U^0_\zeta(\mathfrak{g})$-module $M$ to $u_\zeta(\mathfrak{b})$ is projective iff it is projective when restricted to every Nakayama subalgebra $u_\zeta(f_\alpha):=\langle F_\alpha\rangle$, where $\alpha\in \Phi^+$.
\end{prop}

\section{Webb's Theorem and Periodic components}

In this section we apply a theorem by Kerner and Zacharia to restrict the possible shapes of components to those arising from Euclidean or infinite Dynkin diagrams in the following fashion: For such a diagram $\Delta$ fix an orientation (for $\tilde{\mathbb{A}}_n$ a non-oriented cycle). Then the diagram $\mathbb{Z}[\Delta]$ has vertices $(i,j)$, where $i\in \mathbb{Z}$ and $j\in \Delta$ and arrows $(i,j)\to (i,j')$ and $(i-1,j')\to (i,j)$ for each arrow $j\to j'$ in $\Delta$ and all $i\in \mathbb{Z}$.\\

That all components have a particular shape mainly follows from the fact that $U_\zeta(B_r)$ is an (fg)-Hopf algebra in the following sense:\\

A finite dimensional Hopf algebra $A$ is called (fg)-Hopf algebra if the even cohomology ring $H^{ev}(A,k)$ is finitely generated and each $\Ext^\bullet(M,N)$ is finitely generated as a module for the even cohomology ring (via cup product).\\

For general theory on (fg)-Hopf algebras we refer the reader to \cite{FW09}, \cite{K12} and also to the group algebra case \cite{B91}. Their definition is designed to formulate a theory of support varieteies that has essential features arising in the context of finite groups. To every module $M$ one can associate a variety, the variety associated to the ideal $\ker \Phi_M$, where $\Phi_M:H^{ev}(A,k)\to \Ext^\bullet(M,M)$ is the map induced by tensoring with $M$. This so called support variety $\mathcal{V}_A(M)$ detects certain properties of the module $M$.

\begin{thm}\label{webb}
The non-periodic components $\mathcal{C}$ (i.e. there does not exist $M\in \mathcal{C}$, $m\in \mathbb{N}$ such that $\tau^m M\cong M$) for $u_\zeta(\mathfrak{b})$ are of the form $\mathbb{Z}[\Delta]$, where $\Delta$ is a Euclidean or infinite Dynkin diagram.
\end{thm}

\begin{proof}
The result follows from the fact that $u_\zeta(\mathfrak{b})$ satisfies (fg) (see e.g. \cite[Theorem 6.2.6]{Dru10}). This in particular implies that the complexity of each module is finite. Hence the result follows from \cite[Main Theorem]{KZ11}.
\end{proof}

For the algebras $u_\zeta(\mathfrak{n})$ this also holds since projective modules stay projective under restriction.

\begin{thm}
The non-periodic components for the algebra $u_\zeta(\mathfrak{n})$ are of the form $\mathbb{Z}[\Delta]$, where $\Delta$ is a Euclidean or infinite Dynkin diagram.
\end{thm}

\begin{proof}
The results follows from the fact that the restriction of a projective $u_\zeta(\mathfrak{b})$-module is projective for $u_\zeta(\mathfrak{n})$ by \cite[Remark 3.4]{Dru11}:  For $M\in \modu u_\zeta(\mathfrak{n})$ take a minimal $u_\zeta(\mathfrak{b})$-projective resolution $P_*$ of the induced module $u_\zeta(\mathfrak{b})\otimes_{u_\zeta(\mathfrak{n})} M$. The rate of growth of this projective resolution is finite as $u_\zeta(\mathfrak{b})$ is a finite dimensional Hopf algebra satisfying (fg). If we restrict this resolution to $u_\zeta(\mathfrak{n})$ we will get a projective resolution of $M$ which has finite rate of growth as $M$ is a direct summand of the restriction of $u_\zeta(\mathfrak{b})\otimes_{u_\zeta(\mathfrak{n})} M$ by \cite[Proposition 1.8]{RR85}. Therefore the complexity of $M$ is finite. So the result again follows from \cite[Main Theorem]{KZ11}.
\end{proof}

We proceed by narrowing down the possibilities of periodic components that by a classical result of Happel, Preiser and Ringel are always finite, or infinite tubes (i.e. of the form $\mathbb{Z}[\mathbb{A}_\infty]/\tau^m$ for some $m\in \mathbb{N}$). The case of finite components does not occur in our context since finite components correspond to representation-finite blocks by a classical theorem of Auslander, these do not appear as shown in \cite{FW09, FW11}.

\begin{prop}\label{omegaperiod2}
For $\charac k=0$ let $\ell$ be good for $\Phi$ (i.e. $\ell \geq 3$ for type $\mathbb{B}_n$, $\mathbb{C}_n$ and $\mathbb{D}_n$,
$\ell \geq 5$ for type $\mathbb{E}_6$, $\mathbb{E}_7$ and $\mathbb{G}_2$ and $\ell \geq 7$ for $\mathbb{E}_8$) and $\ell>3$ for types $\mathbb{B}$ and $\mathbb{C}$ and $\ell \nmid n+1$ for type $\mathbb{A}_n$ and $\ell\neq 9$ for $\mathbb{E}_6$. For $\charac k=p>0$ let $p$ be good for $\Phi$ and $\ell\geq h$. Then the $\Omega$-period of every periodic module for $u_\zeta(\mathfrak{b})$ divides $2$ while the $\tau$-period divides $\ell$. Furthermore every $u_\zeta(\mathfrak{n})$-module of complexity $1$ is $\Omega_{u_\zeta(\mathfrak{n})}$- and $\tau_{u_\zeta(\mathfrak{n})}$-periodic and the $\Omega_{u_\zeta(\mathfrak{n})}$-period of a module $M$ divides $2$ while the $\tau$-period is $1$.
\end{prop}

\begin{proof}
By \cite[Proposition 1.1 (iii)]{K12} that the $\Omega_{u_\zeta(\mathfrak{b})}$-period of a module is two follows from the fact that the even cohomology ring is generated in degree two by \cite[Theorem 5.1.3]{Dru11}.\\
By a similar reasoning as in \cite[Section 5]{K12} one can show that $u_\zeta(\mathfrak{b})$ is a $\gamma$-Frobenius extension of $u_\zeta(f_\beta):=\langle F_\beta\rangle$ for some automorphism $\gamma$. As the latter is a Frobenius algebra it follows from \cite[Proposition 1.3]{BF93} that also $u_\zeta(\mathfrak{b})$ is a Frobenius algebra. From this construction we get the Frobenius homomorphism induced by $F^{\ell-1}K^{\ell-1}\mapsto 1$. Now the commutation relations from \cite[Proposition 1.7]{dCK90} in the graded algebra yield a Nakayama automorphism with $K_\alpha\mapsto \zeta^{-\sum_{\beta\in \Phi^+}(\alpha,\beta)}K_\alpha$ and $F_\beta\mapsto \zeta^{\sum_{\alpha'\in \Pi}(\alpha',\beta)}F_\beta$. As $\zeta$ is an $\ell$-th root of unity its order divides $\ell$. As for Frobenius algebras we have $\tau\cong \Omega^2\circ \nu$ it follows that the $\tau_{u_\zeta(\mathfrak{b})}$-period divides $\ell$.\\
By \cite[Lemma 5.13]{Sch09} for the skew group algebra situation under observation we have that every $u_\zeta(\mathfrak{n})$-module of complexity $1$ is periodic iff every $u_\zeta(\mathfrak{b})$-module of complexity $1$ is periodic. This is true because $u_\zeta(\mathfrak{b})$ is an (fg)-Hopf algebra.\\
For the $\Omega$ and $\tau$-period of a $u_\zeta(\mathfrak{n})$-module we use a similar reasoning as for $u_\zeta(\mathfrak{b})$. But as $u_\zeta(\mathfrak{n})$ is not a Hopf algebra we have to use Hochschild cohomology instead of ordinary cohomology. We apply \cite[Proposition 5.4]{EHTSS04} to the subalgebra $HH^0(u_\zeta(\mathfrak{n}))\cdot H^{ev}(u_\zeta(\mathfrak{n}),k)^{\mathbb{Z}/(\ell)^n}$ of the Hochschild cohomolgy $HH^\bullet(u_\zeta(\mathfrak{n})$ (compare \cite[Section 5]{FW11}). As $H^{ev}(u_\zeta(\mathfrak{n}),k)^{\mathbb{Z}/(\ell)^n}\cong H^{ev}(u_\zeta(\mathfrak{b}),k)$ as algebras, the finite generation in degree two follows. The arguments of \cite[Section 5]{FW11} carry through in our slightly more general setting to show the appropriate (Fg)-conditions for Hochschild cohomology. Thus the result follows.
\end{proof}

\begin{rmk}
The bounds are sharp, e.g. for $\mathfrak{g}=\mathfrak{sl}_2$ the $\Omega$-period is two while the $\tau$-periods are $\ell$ and $1$ for the Borel and nilpotent part, respectively.
\end{rmk}

\section{Graded Modules}

We proceed by studying the $\mathbb{Z}^n$-graded modules for the small half quantum groups.

\begin{lem}\label{res}
Let $L\in \{\mathfrak{g}, \mathfrak{b},\mathfrak{n}$\}.
\begin{enumerate}[(i)]
\item The category of finite dimensional modules over $u_\zeta(L)U_\zeta^0(\mathfrak{g})$ is a sum of blocks for the category $\modu u_\zeta(L)\#U_\zeta^0(\mathfrak{g})$.
\item The category of finite dimensional $u_\zeta(L)U_\zeta^0(\mathfrak{g})$-modules has almost split sequences.
\item The canonical restriction functor $\modu u_\zeta(L)U_\zeta^0(\mathfrak{g})\to \modu u_\zeta(L)$ sends indecomposables to indecomposables and almost split sequences to almost split sequences.
\item The restriction functor induces a homomorphism $\mathcal{F}: \Gamma_s(u_\zeta(L)U_\zeta^0(\mathfrak{g}))\to \Gamma_s(u_\zeta(L))$ of stable translation quivers and components are mapped to components via this functor.
\end{enumerate}
\end{lem}

\begin{proof}
The result follows as in \cite[Section 5]{K12}
\end{proof}

An analogue of Webb's Theorem also holds for the category of graded modules for small quantum groups and their half analogues by constructing a subadditive function on the components. Our reasoning is similar to \cite{Fa05} for restricted enveloping algebras.

\begin{thm}
The tree classes of the components of the stable Auslander-Reiten quiver of $u_\zeta(L)U^0_\zeta(\mathfrak{g})$, where $L\in \{\mathfrak{g}, \mathfrak{b}, \mathfrak{n}\}$, are Euclidean diagrams or infinite Dynkin diagrams.
\end{thm}

\begin{proof}
By \cite[Theorem 4.3]{Dru11} and the foregoing lemma there exists $\alpha\in \Phi^+$ such that the restriction $\mathcal{F}(M)|_{u_\zeta(f_\alpha)}$ for $[M]\in \Theta$ is not injective. Consider the induced module $M_\alpha:= u_\zeta(L)\otimes_{u_\zeta(f_\alpha)} k$. The function $d_\alpha: \Theta\to \mathbb{N}, [M]\mapsto \dim \Ext^1_{u_\zeta(L)}(M_\alpha, \mathcal{F}(M))$ is a subadditive $\tau_{u_\zeta(L)U_\zeta^0(\mathfrak{g})}$-invariant function for each component of the Auslander-Reiten quiver of $u_\zeta(L)U_\zeta^0(\mathfrak{g})$: By the arguments of \cite[Section 5]{K12} we have that $u_\zeta(L):u_\zeta(f_\alpha)$ is a $\gamma$-Frobenius extension for some automorphism $\gamma$. Therefore we can use Frobenius reciprocity (\cite[Lemma 7]{NT60}) to conclude that:
\[d_\alpha([M])=\dim \Ext^1_{u_\zeta(L)}(M_\alpha,\mathcal{F}(M))=\dim \Ext^{1}_{u_\zeta(f_\alpha)}(k, \mathcal{F}(M)|_{u_\zeta(f_\alpha)})\neq 0.\]
Therefore the function is well-defined.\\
Moreover, we have that $\mathcal{F}(\tau_{u_\zeta(L)U^0_\zeta(\mathfrak{g})}(M))\cong \tau_{u_\zeta(L)}(\mathcal{F}(M))\cong \Omega^2_{u_\zeta(L)}(\mathcal{F}(M)^{(\nu)})$ by the foregoing lemma, $M_\alpha^{(\nu)}\cong M_\alpha$ since $\nu|_{u_\zeta(f_\alpha)}=\id$ ($\nu\otimes \id$ provides an isomorphism)  and $\Omega^2_{u_\zeta(L)}(M_\alpha)\cong M_\alpha\oplus P$ for a projective module $P$ by inducing a projective $u_\zeta(f_\alpha)$-resolution of $k$ to $u_\zeta(L)$. Therefore we obtain:
\begin{align*}
d_\alpha([\tau_{u_\zeta(L)U_\zeta^0(\mathfrak{g})}(M)])&=\dim \Ext^1_{u_\zeta(L)}(M_\alpha, \mathcal{F}(\tau_{u_\zeta(L)U_\zeta^0(\mathfrak{g})}(M)))\\
&=\dim \Ext^1_{u_\zeta(L)}(\Omega^2_{u_\zeta(L)}(M_\alpha), \Omega^2_{u_\zeta(L)}(\mathcal{F}(M)^{(\nu)}))\\
&=\dim \Ext^1_{u_\zeta(L)}(M_\alpha, \mathcal{F}(M)^{(\nu)})\\
&=\dim \Ext^1_{u_\zeta(L)}(M_\alpha^{(\nu)}, \mathcal{F}(M)^{(\nu)})=d_\alpha([M]),
\end{align*}
i.e. $d_\alpha$ is $\tau$-invariant. It follows from \cite[Lemma 3.2]{ES92} that $d_\alpha$ is a subadditive function on the components of the stable Auslander-Reiten quiver of $u_\zeta(L)U^0_\zeta(\mathfrak{g})$. Hence the result follows from \cite[Theorem, p.286]{HPR80}. The case of finite components does not occur as they restrict to finite components on $u_\zeta(L)$ which are not possible since by \cite{K11} there are no representation-finite blocks for $u_\zeta(L)$.
\end{proof}

\begin{rmk}
This provides a different way to prove the result for components of $u_\zeta(L)$ consisting of gradable modules: The function which is obtained by removing the forgetful functor $\mathcal{F}$ in the foregoing proposition provides a subadditive $\tau_{u_\zeta(L)}$-invariant function for every component of the stable Auslander-Reiten quiver of $u_\zeta(L)$ containing gradable modules. The foregoing proof applies verbatim (replacing $u_\zeta(L)U_\zeta^0(\mathfrak{g})$ with $u_\zeta(L)$ and removing $\mathcal{F}$ everywhere). Thus this provides us with a subadditive $\tau$-invariant function on $\Theta$, a component of $u_\zeta(L)U^0_\zeta(\mathfrak{g})$, that stays subadditive $\tau$-invariant upon restriction. This way one can get information on $\Theta$ or $\mathcal{F}(\Theta)$ by observing the other.
\end{rmk}

\section{Euclidean components}

In this section we exclude components of the form $\mathbb{Z}[\Delta]$, where $\Delta$ is a Euclidean diagram. Our approach relies on results by Scherotzke \cite{Sch09} who corrected and generalized results of Farnsteiner for restricted enveloping algebras \cite{Fa99b}. To apply it we need the following statement:

\begin{prop}
For $\charac k=0$ let $\ell$ be good for $\Phi$ and $\ell>3$ for types $\mathbb{B}$ and $\mathbb{C}$ and $\ell \nmid n+1$ for type $\mathbb{A}_n$ and $\ell\neq 9$ for $\mathbb{E}_6$. For $\charac k=p>0$ let $p$ be good for $\Phi$ and $\ell\geq h$. If $\mathfrak{b}\neq \mathfrak{b}_{\mathfrak{sl}_2}$ there is a non-periodic module of length $3$ for $u_\zeta(\mathfrak{b})$.
\end{prop}

\begin{proof}
Let $\mathcal{S}$ be a set of representatives for the isoclasses of the simple $u_\zeta(\mathfrak{b})$-modules. As $u_\zeta(\mathfrak{b})$ is not a Nakayama algebra, it follows from \cite[Theorem 9]{Hup81} that there is a simple module $S$, such that
\[\sum_{[T]\in \mathcal{S}}\dim \Ext^1_{u_\zeta(\mathfrak{b})}(T,S)\geq 2\]
Let $P$ be the injective hull of $S$. By the definition of $\Ext^1$ and $\soc^2$ we immediately obtain
\[2\leq \sum_{[T]\in \mathcal{S}_{u_\zeta(\mathfrak{b})}}\dim\Hom_{u_\zeta(\mathfrak{b})}(T,\soc^2(P)/\soc(P)),\]
so that $l(\soc^2(P)/\soc P)\geq 2$ and $l(\soc^2(P))\geq 3$.\\
Denote by $\pi:P\to P/\soc(P)$ the natural projection. Then $\pi$ induces a surjection $\soc^2(P)\twoheadrightarrow \soc(P/\soc(P))$ and $\pi(\rad{u_\zeta(\mathfrak{b})}\soc^2(P))=0$. Thus $\rad{u_\zeta(\mathfrak{b})}\soc^2(P)=\soc(P)$. Now let $X\subset \soc^2(P)$ be a submodule of length $3$. Then $\soc X=\soc P$ and $\rad{u_\zeta(\mathfrak{b})}X\subseteq \rad{u_\zeta(\mathfrak{b})}\soc^2(P)= \soc(X)$, so that $\rad{u_\zeta(\mathfrak{b})}X=\soc X$. Thus $\Kopf X=X/\soc X$ has length $2$. Suppose that $X$ was periodic. Then by Proposition \ref{omegaperiod2} $\Omega^2_{u_\zeta(\mathfrak{b})} X\cong X$. As $P$ is the injective hull of $\soc X$ there exists a projective module $Q$ and an exact sequence
\[0\to X\to P\to Q\to X\to 0\]
Therefore $\dim P=\dim Q$, so that $Q$ is indecomposable as all projective indecomposable modules have the same dimension by \cite[p. 85]{DruPhD}. But as $u_\zeta(\mathfrak{b})$ is selfinjective this would mean that $\Kopf X$ is irreducible, therefore has length $1$, a contradiction.
\end{proof}

\begin{thm}
For the stable Aus\-lan\-der-Reiten quivers of $u_\zeta(\mathfrak{b})$ and $u_\zeta(\mathfrak{n})$ if a component is isomorphic to $\mathbb{Z}[\Delta]$ with $\Delta$ Euclidean, then $\Delta$ is $\tilde{\mathbb{D}}_n$, where $n>5$. If we further impose that $\ell>h$ then also these components cannot occur.
\end{thm}

\begin{proof}
We have that $(\mathbb{Z}/\ell)^n$ acts transitively on the simple $u_\zeta(\mathfrak{b})$-modules and all modules of complexity $1$ are $\Omega$ and $\tau$-periodic as $u_\zeta(\mathfrak{b})$ is an (fg)-Hopf algebra and by \ref{omegaperiod2}. By a result of Scherotzke \cite[Theorem 3.3, Theorem 3.7]{Sch09} in this situation the components of Euclidean tree class of $u_\zeta(\mathfrak{b})$ and $u_\zeta(\mathfrak{n})$ can only be of the form $\mathbb{Z}[\tilde{\mathbb{A}}_{12}]$ or $\mathbb{Z}[\tilde{\mathbb{D}}_n]$, where $n>5$. In the case of $\mathbb{Z}[\tilde{\mathbb{A}}_{12}]$ the projective module would have dimension 4 by a result of Erdmann \cite[Theorem IV.3.8.3]{Er90} which is not possible as the dimension of each projective module is divisible by $\ell$. But since by the foregoing proposition there is a non-periodic module of length \cite[Theorem 3.11]{Sch09} tells us that the remainding case of tree class $\tilde{\mathbb{D}}_n$ can not occur because in this case the projective module would have dimension $8$.\\
For $u_\zeta(\mathfrak{n})$ note that $u_\zeta(\mathfrak{b})$ is a skew group algebra of $u_\zeta(\mathfrak{n})$. Therefore as $u_\zeta(\mathfrak{b})$ has no components of tree class $\tilde{\mathbb{D}}_n$ or $\tilde{\mathbb{A}}_{12}$, also $u_\zeta(\mathfrak{n})$ has no components of tree class $\tilde{\mathbb{D}}_n$ or $\tilde{\mathbb{A}}_{12}$ by \cite[Theorem 5.15]{Sch09}.\\
It remains to consider the case of a component $\Theta$ of type $\mathbb{Z}[\tilde{\mathbb{A}}_n]$. For $u_\zeta(\mathfrak{b})$ let $G=(\mathbb{Z}/\ell)^n$ for $u_\zeta(\mathfrak{n})$ let $G=\{e\}$. Let $U=u_\zeta(\mathfrak{b})$, respectively $U=u_\zeta(\mathfrak{n})$. Such a component is attached to a principal indecomposable module (see e.g. \cite{KZ11}). For $\lambda\in G$ let $k_\lambda$ be the simple module corresponding to $\lambda$. Since $\Omega\Theta\cong \Theta$ there exists $\lambda\in G$, such that the automorphism $\varphi: \Gamma_s(U)\to \Gamma_s(U), [M]\mapsto [\Omega(M\otimes k_\lambda)]$ satisfies $\varphi(\Theta)=\Theta$. Note that $(\varphi|_\Theta)^{2\ell}=\Omega^{2\ell}=\tau^\ell$. In \cite[Lemma 2.2]{Fa99b} Farnsteiner has computed the automorphism group of $\mathbb{Z}[\tilde{\mathbb{A}}_n]$, it is $\{\tau^q\circ \alpha^r| q\in \mathbb{Z}, 0\leq r\leq n-1\}$, where $\alpha$ is defined via $\alpha(s,[i])=(s,[i+1])$ with the vertex set of $\mathbb{Z}[\tilde{\mathbb{A}}_n]$ identified with $\mathbb{Z}\times \mathbb{Z}/(n)$. In this notation we write $\varphi|_\Theta=\tau^j\circ \alpha^r$ and obtain $\tau^\ell=\tau^{2\ell j}\circ \alpha^{2\ell r}$, whence $\tau^{\ell (2j-1)}=\alpha^{-2\ell r}$, a contradiction.
\end{proof}

\section{Infinite Dynkin Tree class}

In the last section we have seen that Euclidean tree classes do not occur for components of the stable Auslander-Reiten quiver containing $u_\zeta(\mathfrak{b})U^0_\zeta(\mathfrak{g})$-modules. In this section we want to exclude two of the remaining possible tree classes so that only one possible tree class remains. Our proof follows the strategy of Erdmann for local restricted enveloping algebras in \cite{Er96}.

\begin{prop}\label{dimcongell}
Let $\charac k$ be odd or zero and good for $\Phi$. Let $L\in \{\mathfrak{b},\mathfrak{n}\}$. Assume $\ell\geq h$ and $\mathfrak{g}\neq \mathfrak{sl}_2$. If $M$ is a projective module or a periodic $u_\zeta(L)$-module, which is also a $u_\zeta(\mathfrak{n})U^0_\zeta(\mathfrak{g})$-module, then $\dim M\equiv 0\mod \ell$.
\end{prop}

\begin{proof}
If $M$ is projective, it stays projective when restricting to $u_\zeta(f_\alpha)$ by \cite[Theorem 4.3]{Dru10}. Hence its dimension is divisible by $\ell$. If $M$ is periodic, then by \cite[Proposition 1.1]{K12}, $\mathcal{V}_{u_\zeta(\mathfrak{b})}(M)$ is a line. Note that by a generalization of the Theorem of Ginzburg and Kumar we have that the support variety of a $u_\zeta(\mathfrak{b})$-module identifies with a conical subvariety of $\mathfrak{n}$ (for our restrictions on the parameter see \cite[Theorem 5.2]{Dru11}). Thus there exists $\alpha\in \Phi^+$, such that $f_\alpha\notin \mathcal{V}_{u_\zeta(\mathfrak{b})}(M)$. Hence by \cite[Corollary 5.13]{Dru10} we have that $M|_{u_\zeta(f_\alpha)}$ is projective, hence $\dim M\equiv 0\mod \ell$.
\end{proof}

\begin{thm}
Let $\charac k$ be odd or zero and good for $\Phi$. Assume $\ell\geq h$. The category of $u_\zeta(\mathfrak{n})$-modules, that are also $u_\zeta(\mathfrak{n})U^0_\zeta(\mathfrak{g})$-modules does not contain any components of type $\mathbb{Z}[\mathbb{A}_\infty^\infty]$ or $\mathbb{Z}[\mathbb{D}_\infty]$.
\end{thm}

\begin{proof}
By \cite[Lemma 2.5]{Fa00b}, which is valid in the context of selfinjective algebras, if we have a component $\Theta$ of type $\mathbb{A}_\infty^\infty$ or $\mathbb{D}_\infty$, then there exists an irreducible map $\psi$ corresponding to an arrow in $\Theta$, such that $\psi$ and $\Omega^{-1}\psi$ are surjective or $\psi$ and $\Omega\psi$ are injective.\\
Without loss of generality consider the case that they are surjective otherwise dual arguments yield the result. Hence there is a non-split sequence which is not almost split $0\to M\to E\stackrel{\psi}{\to} N\to 0$. Let $\alpha$ be such that $M|_{u_\zeta(f_\alpha)}$ is not projective (This is possible by \cite[Theorem 4.3]{Dru10}). Let $M|_{u_\zeta(f_\alpha)}\cong M_1^{n_1}\oplus\dots\oplus M_\ell^{n_\ell}$ where $M_i$ is the indecomposable $u_\zeta(f_\alpha)$-module of dimension $i$ and $n_i\geq 0$. Then $\sum_{i=1}^{\ell-1}n_i\neq 0$ as $M$ is not projective. Let $V:=(u_\zeta(\mathfrak{n})\otimes_{u_\zeta(f_\alpha)} M_1)^*$, then $V$ is $\Omega$-periodic: Induction of a projective resolution of the $\Omega$-periodic $u_\zeta(f_\alpha)$-module $M_1$ yields a periodic projective resolution of the induced module $V$. Therefore $V$ has complexity smaller or equal to one and hence is a direct sum of a $\Omega$-periodic (with period two by Proposition \ref{omegaperiod2}) and a projective module.\\
Let $W:=V$ or $W:=\Omega^{-1}V$. Then we have the following chain of isomorphisms since $u_\zeta(\mathfrak{n}):u_\zeta(f_\alpha)$ is a $\gamma$-Frobenius extension:
\[\underline{\Hom}_{u_\zeta(\mathfrak{n})}(M,W)\cong \underline{\Hom}_{u_\zeta(\mathfrak{n})}(W^*,M^*)\cong \underline{\Hom}_{u_\zeta(f_\alpha)}(M_i, M^*)\cong \underline{\Hom}_{u_\zeta(f_\alpha)}(M,M_i^*),\]
where $i=1,\ell-1$ depending on whether $W=V,\Omega^{-1}V$ respectively. This space has dimension $\geq \sum_{i=1}^{\ell-1}n_i$ by the theory of the algebra $u_\zeta(f_\alpha)\cong K[X]/X^\ell$. Moreover $\soc W$ is simple since
\[\Hom_{u_\zeta(\mathfrak{n})}(k,W)\cong \Hom_{u_\zeta(\mathfrak{n})}(W^*,k^*)\cong \Hom_{u_\zeta(f_\alpha)}(M_i,k)\]
which is one-dimensional. In particular $V$ is an indecomposable periodic module.\\
We will show next that $\soc M$ is also simple. We do this by applying \cite[Proposition 1.5]{Er95} to the $\Omega$-periodic module $V$. Hence we either have an embedding $M\to V$ or every map $M\to \Omega^{-1}(V)$ which does not factor through a projective module is a monomorphism. If there is a monomorphism $M\to V$ then since $\soc V$ is simple, so is $\soc M$. Otherwise take a non-zero element of $\underline{\Hom}_{u_\zeta(\mathfrak{n})}(M,\Omega^{-1}(V))$; then a representative must be a monomorphism. Moreover $\soc \Omega^{-1}V$ is simple and so is $\soc M$.\\
This will imply that for all $j\in \mathbb{Z}$, the socle of $\Omega^j(M)$ is simple. As $\Omega$ induces an automorphism of the stable Auslander-Reiten quiver we have an almost split sequence $0\to \Omega^j M\to \Omega^j E\oplus P\to \Omega^j N\to 0$ where $P$ is projective. If $P=0$ then $\Omega^j\psi$ is an irreducible epimorphism and the socle of $\Omega^j M$ is simple by the above argument, replacing $M$ with $\Omega^j M$. Otherwise the sequence is of the form $0\to \rad{P}\to \rad{P}/\soc(P)\oplus P\to P/\soc(P)\to 0$, in particular the map $\Omega^j \psi$ is an irreducible monomorphism with cokernel isomorphic to $\soc \Omega^j M\cong \soc P$, which is simple.\\
Concluding and taking into account that $u_\zeta(\mathfrak{n})$ is local we have a minimal injective resolution of the form
$0\to M\to u_\zeta(\mathfrak{n})\to u_\zeta(\mathfrak{n})\to \dots$.\\
Hence the complexity of $M$ is $\leq 1$ and therefore $M$ is periodic; in particular $\dim M\equiv 0\mod \ell$ by the foregoing proposition.\\
Now $\dim M+\dim \Omega^{-1}M=\dim u_\zeta(\mathfrak{n})$ by the minimal injective resolution. We may assume that $\dim M\geq \frac{1}{2}\dim u_\zeta(\mathfrak{n})$, otherwise replace $M$ by $\Omega^{-1}M$, which is possible since by the choice of $\psi$ we can do the same argument with $\Omega^{-1}\psi$. Now
\[\dim V=\frac{1}{\ell}\dim u_\zeta(\mathfrak{n})< \frac{1}{2}\dim u_\zeta(\mathfrak{n})\leq \dim M.\]
So there is no monomorphism $M\to V$. Therefore every homomorphism $M\to \Omega^{-1} V$ that does not factor through a projective module is injective by \cite[Proposition 1.5]{Er95}. By the above argument we have that $\dim \underline{\Hom}_{u_\zeta(\mathfrak{n})}(M,\Omega^{-1}V)\geq \sum_{i=1}^{\ell-1}n_i$. This must be at least $2$; otherwise we would have that $M|_{u_\zeta(f_\alpha)}$ has a unique non-projective summand and the dimension of $M$ would not be divisible by $\ell$. Hence there are $\phi_1,\phi_2\in \Hom_{u_\zeta(\mathfrak{n})}(M,\Omega^{-1}V)$ with $[\phi_1], [\phi_2]\in \underline{\Hom}_{u_\zeta(\mathfrak{n})}(M,\Omega^{-1}V)$ linearly independent. By the above reasoning we have that $\phi_1$ and $\phi_2$ must be monomorphisms.\\
We know that $\soc M$ is simple, hence there is some $c\in k$ such that $\phi_1-c\phi_2$ is not a monomorphism, and again by \cite[Proposition 1.5]{Er95} we have that $\phi_1-c\phi_2$ factors through a projective module, a contradiction.
\end{proof}

As a corollary we also get the corresponding statement for the Borel part:

\begin{thm}
The category of gradable $u_\zeta(\mathfrak{b})$-modules does not contain any components of type $\mathbb{Z}[\mathbb{A}_\infty^\infty]$ or $\mathbb{Z}[\mathbb{D}_\infty]$.
\end{thm}

\begin{proof}
In the following commutative diagram indecomposable modules are mapped to indecomposable modules and Auslander-Reiten sequences are mapped to Auslander-Reiten sequences by the vertical arrows by Lemma \ref{res}. Therefore also by the horizontal arrow.
\[\begin{xy}
\xymatrix{
\modu u_\zeta(\mathfrak{b}) U_\zeta^0(\mathfrak{g})\ar@{-}@<2pt>[r]\ar@{-}@<-2pt>[r]\ar[d]^{\mathcal{F}} &\modu u_\zeta(\mathfrak{n}) U_\zeta^0(\mathfrak{g})\ar[d]^{\mathcal{F}}\\
\modu u_\zeta(\mathfrak{b})\ar[r]^{\res} &\modu u_\zeta(\mathfrak{n})
}
\end{xy}
\]
Thus the existence of an Auslander-Reiten sequence with three indecomposable direct summands would be preserved and therefore the non-existence of such sequences for $u_\zeta(\mathfrak{n})$ implies the same for $u_\zeta(\mathfrak{b})$, i.e. there are no components of type $\mathbb{Z}[\mathbb{D}_\infty]$. Furthermore as for every component $\Theta$ containing a gradable module there is an Auslander-Reiten sequence in $\Theta$ with indecomposable middle term for $u_\zeta(\mathfrak{n})$ the same has to hold for $u_\zeta(\mathfrak{b})$ and thus the case of components of type $\mathbb{Z}[\mathbb{A}_\infty^\infty]$ can be excluded.
\end{proof}

\begin{rmk}
The same argument also works in the case of gradable modules for the restricted enveloping algebra of a Borel subalgebra.
\end{rmk}

\section*{Acknowledgement}
The results of this article are part of my doctoral thesis, which I am currently writing at the University of Kiel. I would like to thank my advisor Rolf Farnsteiner for his continuous support, especially for answering my questions on the classical case. I also would like to thank Chris Drupieski for answering some of my questions on Frobenius-Lusztig kernels, especially on restrictions of the parameter. Furthermore I thank the members of my working group for proof reading.

\bibliographystyle{alpha}
\bibliography{publication}

\newcommand{\etalchar}[1]{$^{#1}$}
\begin{thebibliography}{EHT{\etalchar{+}}04}

\bibitem[ARS95]{ARS95}
Maurice Auslander, Idun Reiten, and Sverre~Olaf Smal\o{}.
\newblock {\em Representation Theory of Artin Algebras}.
\newblock Cambridge University Press, 1995.

\bibitem[ASS06]{ASS06}
Ibrahim Assem, Daniel Simson, and Andrzej Skowro\'nski.
\newblock {\em Elements of the Representation Theory of Associative Algebras,
  Volume I: Techniques of Representation Theory}.
\newblock Cambridge University Press, 2006.

\bibitem[Ben91]{B91}
David~John Benson.
\newblock {\em Representations and cohomology, II: Cohomology of groups and
  modules}.
\newblock Cambridge University Press, 1991.

\bibitem[BF93]{BF93}
Allen~Davis Bell and Rolf Farnsteiner.
\newblock On the theory of {F}robenius extensions and its application to {L}ie
  superalgebras.
\newblock {\em Transactions of the American Mathematical Society},
  335(1):407--424, 1993.

\bibitem[Cib97]{Cib97}
Claude Cibils.
\newblock Half-quantum groups at roots of unity, path algebras, and
  representation type.
\newblock {\em International Mathematics Research Notices}, (12):541--553,
  1997.

\bibitem[Dru09]{DruPhD}
Christopher~Martin Drupieski.
\newblock {\em Cohomology of {F}robenius-{L}usztig Kernels of Quantized
  Enveloping Algebras}.
\newblock PhD thesis, University of Virginia, 2009.

\bibitem[Dru10]{Dru10}
Christopher~Martin Drupieski.
\newblock On injective modules and support varieties for the small quantum
  group.
\newblock {\em International Mathematics Research Notices},
  2011(10):2263--2294, 2010.

\bibitem[Dru11]{Dru11}
Christopher~Martin Drupieski.
\newblock Representations and cohomology for {F}robenius-{L}usztig kernels.
\newblock {\em Journal of Pure and Applied Algebra}, 215(6):1473--1491, 2011.

\bibitem[EHT{\etalchar{+}}04]{EHTSS04}
Karin Erdmann, Miles Holloway, Rachel Taillefer, Nicole Snashall, and {\O}yvind
  Solberg.
\newblock Support varieties for selfinjective algebras.
\newblock {\em $K$-Theory}, 33(1):67--87, 2004.

\bibitem[Erd90]{Er90}
Karin Erdmann.
\newblock {\em Blocks of tame representation type and related algebras}, volume
  1428 of {\em Lecture Notes in Mathematics}.
\newblock Springer-Verlag, Berlin, 1990.

\bibitem[Erd95]{Er95}
Karin Erdmann.
\newblock On {A}uslander-{R}eiten components for group algebras.
\newblock {\em Journal of Pure and Applied Algebra}, 104:149--160, 1995.

\bibitem[Erd96]{Er96}
Karin Erdmann.
\newblock The {A}uslander-{R}eiten quiver of restricted enveloping algebras.
\newblock In {\em Representation theory of algebras ({C}ocoyoc, 1994)},
  volume~18 of {\em CMS Conf. Proc.}, pages 201--214. American Mathematical
  Society, Providence, RI, 1996.

\bibitem[ES92]{ES92}
Karin Erdmann and Andrzej Skowro\'nski.
\newblock On {A}uslander-{R}eiten components of blocks and self-injective
  special biserial algebras.
\newblock {\em Trans. Amer. Math. Soc.}, 330:165--189, 1992.

\bibitem[Far99]{Fa99b}
Rolf Farnsteiner.
\newblock On {A}uslander-{R}eiten quivers of enveloping algebras of restricted
  {L}ie algebras.
\newblock {\em Mathematische Nachrichten}, 202:43--66, 1999.

\bibitem[Far00]{Fa00b}
Rolf Farnsteiner.
\newblock Auslander-{R}eiten components for {L}ie algebras of reductive groups.
\newblock {\em Advances in Mathematics}, 155(1):49--83, 2000.

\bibitem[Far05]{Fa05}
Rolf Farnsteiner.
\newblock Auslander-{R}eiten components for {$G_1T$}-modules.
\newblock {\em Journal of algebra and its applications}, 4(6):739--759, 2005.

\bibitem[FW09]{FW09}
Jörg Feldvoss and Sarah Witherspoon.
\newblock Support varieties and representation type of small quantum groups.
\newblock {\em International Mathematics Research Notices}, 2010(7):1346--1362,
  2009.

\bibitem[FW11]{FW11}
Jörg Feldvoss and Sarah Witherspoon.
\newblock Support varieties and representation type of self-injective algebras.
\newblock preprint, May 2011.

\bibitem[HPR80]{HPR80}
Dieter Happel, Udo Preiser, and Claus~Michael Ringel.
\newblock {V}inberg's characterization of {D}ynkin diagrams using subadditive
  functions with applications to {$D\Tr$}-periodic modules.
\newblock In {\em Representation Theory, {II} ({P}roceedings of the {S}econd
  {I}nternational {C}onference on Representations of Algebras, {C}arleton
  {U}niversity, {O}ttawa, {O}nt., 1979}, volume 832 of {\em Lecture Notes in
  Mathematics}, pages 280--294. Springer, Berlin, 1980.

\bibitem[Hup81]{Hup81}
Lillian Elizabeth~Peters Hupert.
\newblock Homological characteristics of pro-uniserial rings.
\newblock {\em Journal of Algebra}, 69(1):43--66, 1981.

\bibitem[Kül11]{K11}
Julian Külshammer.
\newblock Representation type of {F}robenius-{L}usztig kernels.
\newblock To appear in: The Quarterly Journal of Mathematics, 2011.

\bibitem[Kül12]{K12}
Julian Külshammer.
\newblock Auslander-{R}eiten theory of {F}robenius-{L}usztig kernels.
\newblock preprint, January 2012.

\bibitem[KZ11]{KZ11}
Otto Kerner and Dan Zacharia.
\newblock Auslander-{R}eiten theory for modules of finite complexity over
  self-injective algebras.
\newblock {\em Bulletin of the London Mathematical Society}, 43(1):44--56,
  2011.

\bibitem[NT60]{NT60}
Tadasi Nakayama and Tosiro Tsuzuku.
\newblock On {F}robenius extensions {I}.
\newblock {\em Nagoya Mathematical Journal}, 17:89--110, 1960.

\bibitem[RR85]{RR85}
Idun Reiten and Christine Riedtmann.
\newblock Skew group algebras in the representation theory of {A}rtin algebras.
\newblock {\em Journal of Algebra}, 92(1):224--282, 1985.

\bibitem[Sch09]{Sch09}
Sarah Scherotzke.
\newblock Euclidean components for a class of self-injective algebras.
\newblock {\em Colloquium Mathematicum}, 115(2):219--245, 2009.

\end{thebibliography}

\end{document}